\documentclass[a4paper, 11pt, reqno]{amsart}
\usepackage[usenames,dvipsnames]{color}
\usepackage{amssymb, amsmath, latexsym, graphics, graphicx}
\newtheorem{theorem}{Theorem}[section]
\newtheorem{lemma}[theorem]{Lemma}
\newtheorem{corollary}[theorem]{Corollary}
\newtheorem{proposition}[theorem]{Proposition}

\theoremstyle{definition}
\newtheorem{definition}[theorem]{Definition}
\newtheorem{example}[theorem]{Example}

\theoremstyle{definition}
\newtheorem{remark}[theorem]{Remark}
\makeatother

\def\rk{{\rm rk}\,}

\def\S{{\mathcal S}}
\def\M{{\mathcal M}}
\def\R{{\mathcal R}}
\def\L{{\mathcal L}}
\def\D{{\mathcal D}}
\def\H{{\mathcal H}}
\def\J{{\mathcal J}}

\def\Im{{\rm Im}}

\begin{document}
\title{Linear functions preserving green's relations over fields}
\author[A. Guterman, M. Johnson, M. Kambites, A. Maksaev]{ALEXANDER GUTERMAN $^{1,2,3}$,  
	MARIANNE JOHNSON $^{4}$, \\
	MARK KAMBITES $^{4}$ and ARTEM MAKSAEV $^{1,3}$}
\thanks{\hspace{-0.5cm}The work of the first and the fourth authors is supported by RSF grant 17-11-01124 \\ \\ $^{1}$Faculty of Mathematics and Mechanics,
	Moscow State University,
	Moscow, GSP-1, 119991, Russia;\\$^{2}$
	Moscow Institute of Physics and Technology,
	Dolgoprudny, 141701, Russia\\ $^3$  Moscow Center for Fundamental and Applied Mathematics, Moscow, 119991, Russia \\ $^4$ Department of Mathematics, University of Manchester,
	Manchester M13 9PL, UK \\ \\
Email: \texttt{guterman@list.ru} (Alexander Guterman), \\
\texttt{Marianne.Johnson@manchester.ac.uk} (Marianne Johnson), \\
\texttt{Mark.Kambites@manchester.ac.uk} (Mark Kambites), \\
\texttt{artmak95@mail.ru} (Artem Maksaev).}
\maketitle

\begin{abstract}
We study linear functions on the space of $n \times n$ matrices over a field which preserve or strongly preserve each of Green's equivalence relations ($\L$, $\R$, $\H$ and $\J$) and the corresponding pre-orders. For each of these relations we are able to completely describe all preservers over an algebraically closed field (or more generally, a field in which every polynomial of degree $n$ has a root), and all strong preservers and bijective preservers over any field. Over a general field, the non-zero $\J$-preservers are all bijective and coincide with the bijective rank-$1$ preservers, while the non-zero $\H$-preservers turn out to be exactly the invertibility preservers, which are known.
The $\L$- and $\R$-preservers over a field with ``few roots'' seem harder to describe: we give a family of examples showing that they can be quite wild. 


\end{abstract}

\section{Introduction}

The investigation of linear transformations preserving natural functions, invariants and relations on matrices has a long history, dating back to a result of Frobenius \cite{Fr} describing maps which preserve the determinant. The maps preserving  minors of a certain order $r$ were characterised by Schur \cite{Sch}, while the singularity  preservers were described by Dieudonn\'{e}~\cite{Di2}. The characterisation of invertibility preservers goes back to the famous Kaplansky problem; see \cite{Semrl} for details. The results for complex matrices are due to Marcus and Purves \cite{MarcusPurves}, who proved that every linear unital invertibility preserving map on square complex matrics is either an inner automorphism or inner anti-automorphism. The same result does not hold if we replace complex matrices by real matrices. Indeed, the map $T$ defined by $T\left(\begin{smallmatrix} a & b \\ c & d \end{smallmatrix}\right)= \left(\begin{smallmatrix} a & b \\ -b & a \end{smallmatrix}\right)$ is linear, unital, and preserves invertibility, but is not an automorphism or anti-automorphism.  More recent work of de Seguins Pazzis \cite{dpazzis}
has given a characterisation of the invertibility preservers over an arbitrary field.
In general, understanding the linear maps which preserve a given relation helps to better understand the relation, and also to produce new sets of related elements from known ones. For further results and applications see the survey~\cite{P} and the monograph \cite{Molnar}.

\textit{Green's relations} are a number of equivalence relations and pre-orders which are defined upon any semigroup; they encapsulate the ideal structure of the semigroup, and play a central role in almost every aspect of semigroup theory. In particular, they are natural relations to define upon the set of $n \times n$ matrices over any field (or indeed, any semiring), when viewed as a semigroup under matrix multiplication. 
 It is therefore natural to consider maps which preserve Green's relations on matrices.
In \cite{GutermanJohnsonKambites}, motivated by recent interest in the structure of the \textit{tropical semifield}, three of the present authors classified  bijective linear maps
which preserve (or strongly preserve) each of Green's relations on the space of $n \times n$ matrices over an \textit{anti-negative semifield} (that is, any semifield which is not a field - see Proposition~\ref{dichotomy} below).

Much research on semifields attempts in some sense to extend or generalise existing (indeed, often classical) knowledge about fields. The results of \cite{GutermanJohnsonKambites} are rather unusual in this respect since they solved a very natural problem for all semifields \textit{except} fields, while the corresponding question for fields remained open. The purpose of the present paper is to address this problem for fields (and hence, in combination with \cite{GutermanJohnsonKambites}, for semifields in complete generality).

Working over a field, we completely classify the bijective $\L$-preservers (and hence $\L$-order preservers), and we show that strong preservers are automatically bijective so this also serves to classify the strong preservers. In the case that the field is algebraically closed (or more generally, has roots for every polynomial of degree the dimension of the matrices) we are able to classify all $\L$-preservers, but we give examples (in every even dimension over the real numbers, and in every dimension greater than $1$ over the rational numbers) suggesting that $\L$-preservers are hard to classify in general when the field does not have enough roots. Dual results apply to $\R$ and $\R$-order preservers.

For the $\H$ relation, we show that a non-zero linear map preserves $\H$ if and only if it preserves invertibililty; this combines with known results
\cite{dpazzis} to give a complete description over arbitrary fields. We are also able to completely describe the maps which preserve the $\J$-relation (which coincides with the $\D$-relation for semigroups of matrices over fields); it turns out that non-zero $\J$-preservers are all bijective.

Our paper is organised as follows.
In Section 2 we recall some necessary preliminary definitions, results and notions on the topics important for our paper, including semifields, Green's relations and linear preservers. Sections 3, 4, and 5 give results and associated examples relating respectively to the $\L,\R,\leq_\L$, and $\leq_\R$ relations, to the $\H$ and $\leq_\H$ relations, and to the $\J$ (equivalently, $\D$) and $\leq_\J$ relations.

\section{Preliminaries}

While this paper is primarily concerned with matrices over fields, we shall at times refer also to \textit{semifields}. By a \textit{semifield} we mean a set $\S$ equipped with two associative, commutative binary operations $+$ and $\times$ such that $\times$ distributes over $+$, there is an element $0$ which acts as an identity element for $+$ and a zero element for $\times$, and $\S \setminus \lbrace 0 \rbrace$ forms a group under $\times$ with identity element denoted $1$. Thus, a semifield is an algebraic structure which satisfies all of the standard field axioms except perhaps for the presence of additive inverses (that is, negatives). Notable examples, apart from fields, include the \textit{tropical semifield} ($\mathbb{R} \cup \lbrace -\infty \rbrace$ under operations ``$\max$" and $+$) and the \textit{boolean semifield} ($\lbrace \mathrm{TRUE}, \mathrm{FALSE} \rbrace$ under operations ``$\textrm{or}$'' and ``$\textrm{and}$''). We often write multiplication as juxtaposition. An \textit{$\S$-module} is a commutative monoid $M$ (written additively) with an action of $\S$ upon it (written $s \cdot x$ for $s \in \S$ and $x \in M$), satisfying $1\cdot x = x$, $s\cdot (x+y) = (s\cdot x) + (s \cdot y)$, $(s+t) \cdot x = (s \cdot x) + (t\cdot x)$ and $s\cdot (t\cdot x) = (s t)\cdot x$
for all $s, t \in \S$ and $x,y \in M$.

A semifield is termed
\textit{anti-negative} if no non-zero element has an additive inverse (in other words, $a+b=0$ implies $a=b=0$). The following dichotomy is well known but for completeness we include a brief proof.

\begin{proposition}\label{dichotomy}
	Let $\S$ be a semifield. Then either $S$ is a field or $S$ is anti-negative.
\end{proposition}

\begin{proof}
Suppose $\S$ is not anti-negative. Then it has a non-zero element $a$ with an additive inverse, say $b$ such that $a+b=0$. But now for any element $c$ we have $c + a^{-1} b c = a^{-1} ( a + b) c = a^{-1} 0 c = 0$ so $a^{-1} b c$ is an additive inverse for $c$. Thus, $\S$ is a field.
\end{proof}

For  a semifield $\S$ and a natural number $n$, we denote by $M_n(\S)$ the set of $n \times n$ matrices over $\S$, which we view
both as an $\S$-module (which   means just a $K$-vector-space if  $\S = K$ is a field) and a monoid under matrix multiplication. We write
$0_{n \times n}$ for the $n \times n$ zero matrix, which forms a zero element in $M_n(\S)$. We write $GL_n(\S)$ for the subgroup of
multiplicatively invertible matrices (which is the usual general linear group when $\S$ is a field, and the group of monomial matrices
when $\S$ is an anti-negative semifield). For $x_1, \dots, x_n \in \S$ we write ${\rm diag}(x_1, \dots, x_n)$ for the diagonal matrix with entry $x_i$ in the $i$th diagonal position and the zero element of the semifield elsewhere.
For $M \in M_n(\S)$ we write ${\rm Row}_\S(M)$
and ${\rm Col}_\S(M)$ for the row space and column space respectively of $M$ (that is, the $\S$-module generated by the rows or
columns of $M$ viewed as elements of the free $\S$-module $\S^n$). If $\S$ is a field we write $\rk(M)$ for the rank of
$M$ in the usual sense.

 If $\equiv$ is a binary relation on $M_n(\S)$, then a linear function $f : M_n(\S) \to M_n(\S)$ is called:
\begin{itemize}
\item a \textit{(weak) $\equiv$-preserver} if $(A \equiv B) \implies (f(A) \equiv f(B))$ for all \\ $A, B \in M_n(\S)$;
\item a \textit{strong $\equiv$-preserver} if $(A \equiv B) \iff (f(A) \equiv f(B))$ for all \\ $A, B \in M_n(\S)$.
\end{itemize}
We say also that a function $f$ \textit{(weakly) preserves} or \textit{strongly preserves} $\equiv$ with the obvious meaning. For another
binary relation $\cong$ on $M_n(\S)$   we say that $f$ \textit{exchanges} $\equiv$ with $\cong$ if for all $A, B \in M_n(\S)$ we have
$$(A \equiv B) \iff (f(A) \cong f(B))  \ \ \textrm{  and  } \ \ (A \cong B) \iff (f(A) \equiv f(B)).$$
For a set $X \subseteq M_n(\S)$, we say that $f$ \textit{(weakly) preserves} $X$ if $(A \in X) \implies (f(A) \in X)$.

The following relations, which are due to Green \cite{Green}, are of fundamental importance in the study of semigroups. (For general semigroups the
definitions are slightly more complex than those given below, but since all the semigroups we consider here are monoids we are able to use
a slightly simplified form of the definition.)
\begin{definition}
Let $\M$ be a monoid. For $a, b \in \M$, we say that:
\begin{itemize}
\item[(i)] $a \leq_\R b$ if and only if $a\M \subseteq b\M$, that is, if there exists $s \in \M$ with $a = bs$. We say that $a\R b$ if $a \leq_\R b$ and $b \leq_\R a$, or in other words if $a$ and $b$ generate the same principal right ideal of $\M$.

\item[(ii)] $a \leq_\L b$ if and only if $\M a \subseteq \M b$, that is, if there exists $s \in \M$ with $a = sb$. We say that $a\L b$ if $a \leq_\L b$ and $b \leq_\L a$, or in other words if $a$ and $b$ generate the same principal left ideal of $\M$.

\item[(iii)] $a \leq_\J b$ if and only if $\M a \M \subseteq \M b \M$, that is, if there exist $s, t \in \M$ with $a = sbt$. We say that $a\J b$ if $a \leq_\J b$ and $b \leq_\J a$, or in other words if $a$ and $b$ generate the same principal two-sided ideal of $\M$.

\item[(iv)] $a \leq_\H b$ if and only if $a \leq_\R b$ and $a \leq_\L b$. We say that $a\H b$ if $a\R b$ and $a \L b$ (or equivalently, if $a \leq_\H b$ and $b \leq_\H a$).
\end{itemize}
\end{definition}

We note that the relations $\L$, $\R$, $\J$ and $\H$ are equivalence relations, while $\leq_\L$, $\leq_\R$, $\leq_\J$ and $\leq_\H$ are pre-orders.
We are concerned with characterising those bijective linear maps on monoids $M_n(\S)$ which preserve or strongly preserve the relations defined above. For $\S$ an anti-negative semifield (which by Proposition~\ref{dichotomy} means any semifield which is not a field) this was accomplished in \cite{GutermanJohnsonKambites}; the present paper addresses the corresponding question for $\S$ a field, and hence, in combination with \cite{GutermanJohnsonKambites}, for semifields in complete generality.

When $\S = K$ is a field, it is well known (see for example \cite[Chapter 2, Exercise 19]{How}) that the relations $\J$ and $\leq_\J$ are completely determined by \textit{rank}: specifically $A \leq_\J B$ if and only if $\mathrm{rk}(A) \leq \mathrm{rk}(B)$, so $A \J B$ if and only if $\mathrm{rk}(A) = \mathrm{rk}(B)$. A further relation, $\D$, is also usually defined on a monoid $\M$ by  $a\D b$ if and only if there exists $c \in \M$ such that $a \R c$ and $c \L b$, and was considered in \cite{GutermanJohnsonKambites}. However, it is well known (see \cite[Chapter 2, Exercise 19]{How} again) that when $\S = K$ is a field, the relations $\D$ and $\J$ on $M_n(K)$ coincide, so we have no need to consider $\D$ here. The relation we have called $\leq_\H$ is less standard and less widely studied than the others, but since it is natural to define and it is handled by our arguments for $\H$, it makes sense to include it here.
For $A \in M_n(\S)$, we write $\H_A, \L_A, \R_A$ and $\J_A$ to denote the equivalence classes of $A$ with respect to the $\mathcal{H}$-, $\mathcal{L}$-, $\mathcal{R}$-, and $\mathcal{J}$-relations, respectively.

The following celebrated theorem will be one of the main tools in our further investigations. There are many different formulations and proofs going back to the work by L.~K.~Hua \cite[Theorem 2]{Hua}; see also the book  \cite[Chapter 3]{Wan} by Z.-X. Wan. We need the following formulation from \cite{Lautemann}.

\begin{theorem} \label{rank-1-preservers} \cite[Theorem 2]{Lautemann}
Let $K$ be a field and $T\colon M_n(K) \to M_n(K)$ be a bijective linear map that preserves the set of rank $1$ matrices. Then there exist $P, Q \in GL_n(K)$ such that either:
\begin{itemize}
\item $T(A) = PAQ$ for all $A \in M_n(K)$ or
\item $T(A) = PA^TQ$ for all $A \in M_n(K)$.
\end{itemize}
\end{theorem}

The following statement is standard for linear preservers.   
\begin{lemma} \label{transformations}
Let K be a field and $P, Q \in GL_n(K)$. Then
\begin{itemize}
\item[(i)] the bijective linear transformation $T:M_n(K) \rightarrow M_n(K)$ defined by $T(A) = PAQ$ for all $A \in M_n(K)$ preserves each of Green's relations $\mathcal{L}, \mathcal{R}, \mathcal{H}, \mathcal{J}$ and the orders $\leq_\mathcal{L}$, $\leq_\mathcal{R}$, $\leq_\mathcal{H}$, and $\leq_\mathcal{J}$; and
\item[(ii)] the bijective linear transformation $T:M_n(K) \rightarrow M_n(K)$ defined by $T(A) = PA^TQ$ for all $A \in M_n(K)$ preserves the relations $\mathcal{H}, \mathcal{J}$ and the orders $\leq_\mathcal{H}$, and $\leq_\mathcal{J}$. It exchanges $\mathcal{L}$ with $\mathcal{R}$
and exchanges $\leq_\mathcal{L}$ with $\leq_\mathcal{R}$.
\end{itemize}
\end{lemma}

\begin{proof}
(i) It suffices to show that each of the orders is preserved, from which it will follow that the corresponding relation must be preserved.
If $A \leq_{\mathcal{L}} B$, then $A=XB$ for some $X \in M_n(K)$, and so
$$T(A) = PAQ = PXBQ = (PXP^{-1})PBQ = (PXP^{-1})T(B).$$
Similarly, for $\leq_{\mathcal{R}}$, $\leq_{\mathcal{J}}$ and $\leq_{\mathcal{H}}$.

(ii) The transpose map $ A \to A^T$ on $M_n(K) $  is easily seen to exchange $\L$ with $\R$ and $\leq_\L$ with $\leq_R$, from which it follows
that it preserves $\H$ and $\leq_\H$. This map also preserves $\J$ and $\leq_\J$ since it preserves rank.
 The claim now follows from the fact that each map of the given form is the composition
of the transpose map with a map of the form in part (i).
\end{proof}

We shall see below that the transformations described by Lemma~\ref{transformations} are the only \textit{bijective} linear transformations which preserve any of Green's relations. It is more complex to describe the non-bijective linear transformations preserving  the various relations in complete generality. We provide characterisations in the case where the field is algebraically closed.

\section{The $\L$ and $\R$ relations}

We recall a characterisation of Green's $\L$ and $\R$ relations for matrix semigroups. Further we consider the elements of $K^n$ as row vectors $v=(v_1, \ldots, v_n)$.
The following proposition is well known and can be found in the literature for example as \cite[Lemma 2.1]{Okninsk} in the field case or \cite[Proposition 4.1]{HK} in greater generality.

\begin{proposition} \label{Okn}
Let $\S$ be a semifield. For $A, B \in M_n(\S)$, we have:
\begin{itemize}
\item[(i)] $A \L B$ if and only if ${\rm Row}_\S(A) = {\rm Row}_\S(B)$;
\item[(ii)] $A \R B$ if and only if  ${\rm Col}_\S(A) = {\rm Col}_\S(B)$;
\item[(iii)] $A \leq_\L B$ if and only if ${\rm Row}_\S(A) \subseteq {\rm Row}_\S(B)$;
\item[(iv)] $A \leq_\R B$ if and only if ${\rm Col}_\S(A) \subseteq {\rm Col}_\S(B)$.
\end{itemize}
In particular, if $\S = K$ is a field then the set of invertible matrices in $M_n(K)$ forms a single $\L$-class and a single $\R$-class, and hence
also a single $\H$-class (because $\H = \L \cap \R$) and a single $\J$-class (because $\J = \D$ when $\S$ is a field, and $\D$ is the transitive closure of $\L \cup \R$).
\end{proposition}

\subsection{Bijective linear  $\L$ and $\R$ relation preservers}
\begin{lemma} \label{L-properties-rank>1}
Let $K$ be a field, and suppose $V \subseteq M_n(K)$ is both a non-trivial linear subspace and a union of $\L$-classes (or $\R$-classes). Then $V$ contains a rank $1$ matrix.
\end{lemma}
\begin{proof}
We treat the case where $V$ is a union of $\L$-classes, the case where it is a union of $\R$-classes being dual.
Since $V$ is non-trivial, we may choose some non-zero $A \in V$. If $A$ has rank $1$ then we are done. Otherwise, $A$ 
has two linearly independent rows, say, rows $i$ and $j$. Let $P$ denote the permutation matrix swapping $i$ and $j$. Then $(PA) \L A$ so $PA \in V$
(since $V$ is a union of $\L$-classes), and hence $A-PA \in V$ (since $V$ is a linear subspace).  But it is easy to see that $\rk (A-PA)=1$ since its $i$th and $j$th rows differ only in sign whilst all other rows are zero.\end{proof}

\begin{lemma} \label{LR-properties-rank-1}
Let $K$ be a field and $A \in M_n(K)$. The following are equivalent:
\begin{itemize}
\item[(i)] $A$ has rank 1;
\item[(ii)] $\L_A \cup \{0_{n \times n}\}$ is a linear subspace of $M_n(K)$ of dimension $n$;
\item[(iii)] $\R_A \cup \{0_{n \times n}\}$ is a linear subspace of $M_n(K)$ of dimension $n$.
\end{itemize}
\end{lemma}

\begin{proof}
To see that (ii) implies (i), we prove the contrapositive. It is clear that if $A$ has rank  $0$, then $\L_A \cup\{0_{n \times n}\} = \{0_{n \times n}\}$, which is not a subspace of dimension $n$. When $A$ has rank at least 2, $\L_A \cup \{0_{n \times n}\}$ is a union of $\L$-classes containing no rank $1$ matrices, so by Lemma \ref{L-properties-rank>1} it is not a subspace at all.

To see that (i) implies (ii), suppose that $A$ has rank $1$. Thus the row space of $A$ is $1$-dimensional. Recalling from Proposition~\ref{Okn} that two matrices are $\mathcal{L}$-related if and only if they have the same row space, we see that there exists $v \in K^n \smallsetminus \{0\}$ such that
$$\L_A \cup \{0_{n \times n}\} \ = \ \{ u^T v :u \in K^n\}.$$
This is clearly a linear subspace of dimension $n$. 

The equivalence of   (i) and (iii) is dual.
\end{proof}

We are now ready to establish our first main result. In fact, the statement exactly mirrors a result in the anti-negative semifield case from \cite{GutermanJohnsonKambites} and hence applies to semifields in complete generality, although the proof for fields is completely
different.
\begin{theorem} \label{L_maps_bijective_case}
Let $\S$ be a semifield and $T\colon M_n(\S) \rightarrow M_n(\S)$ a bijective linear map. The following are equivalent:
\begin{itemize}
\item[(i)] $T$ preserves $\mathcal{L}$;
\item[(ii)] $T$ preserves $\mathcal{R}$;
\item[(iii)] $T$ preserves $\leq_\mathcal{L}$;
\item[(iv)] $T$ preserves $\leq_\mathcal{R}$;
\item[(v)] there exist $P, Q \in  GL_n(\S)$ such that $T(A) = PAQ$ for all $A \in M_n(\S)$.
\end{itemize}
\end{theorem}

\begin{proof}
The case where $\S$ is an anti-negative semifield is \cite[Theorem 3.5]{GutermanJohnsonKambites}, so by Proposition~\ref{dichotomy}
we may assume $\S$ is a field $K$.
It is clear from the definitions that (iii) implies (i), and (iv) implies (ii). Moreover, it follows immediately from Lemma \ref{transformations} that (v) implies (i)-(iv).
We shall show that (i) implies that rank $1$ is preserved by $T^{-1}$.

Suppose then that $T$ preserves $\mathcal{L}$ and let $B \in M_n(K)$ be a matrix of rank $1$. Write $L$ for the union of all those $\mathcal{L}$-classes which are mapped by $T$ into $\L_B$. Since $T$ is surjective we see that $T(L \cup \{0_{n \times n}\}) = \L_B \cup\{0_{n \times n}\}$. By Lemma \ref{LR-properties-rank-1} the latter is an $n$-dimensional subspace of $M_n(K)$, and so, since  $T$ is a linear isomorphism, we must have that $L \cup\{0_{n \times n}\}$ is an $n$-dimensional subspace of $M_n(K)$ too.  Notice that if every element of $L$ were of rank greater than 1, then Lemma~\ref{L-properties-rank>1} would mean that $L \cup \{0_{n \times n}\}$ was not a subspace, so we may choose a matrix $X \in L$ of rank 1. By Lemma \ref{LR-properties-rank-1} we see that $\L_X \cup \{0_{n \times n}\}$ is an $n$-dimensional subspace of $M_n(K)$. Since $L$ is a union of $\mathcal{L}$-classes notice that $\L_X \cup\{0_{n \times n}\}$ must also be a subspace of $L \cup\{0_{n \times n}\}$. But since both of these spaces have dimension $n$ we conclude that $L = \L_X$. Hence for each rank $1$ matrix $B$ we have shown that $T^{-1}(B)$ is also of rank $1$.

 Therefore (i) implies that $T^{-1}$ preserves the set of rank $1$ matrices.
 From our argument it is clear that replacing $\mathcal{L}$-classes by $\mathcal{R}$-classes throughout also shows that (ii) implies that rank $1$ is preserved by $T^{-1}$.  Since $T$ is bijective, this means that $T$ must also preserve rank $1$, so Theorem \ref{rank-1-preservers} yields that (v) holds.

\end{proof}

\begin{proposition} \label{L_strong}
Let $K$ be a field and $T\colon M_n(K) \rightarrow M_n(K)$ be a linear map which strongly preserves either $\L$ or $\R$. Then $T$ is bijective.
\end{proposition}
\begin{proof}
If $T(A) = 0_{n \times n}$ then $T(A) = T(0_{n \times n})$. Since $T$ strongly preserves $\L$ or $\R$ but $0_{n\times n}$ is neither $\L$- nor $\R$-related to any other matrix, this means $A = 0_{n\times n}$. Thus $T$ must be injective, and since $M_n(K)$ is finite dimensional and $T$ is linear, it must be a bijection by the rank-nullity theorem.
\end{proof}

\subsection{Non-bijective maps}

Note that there exist non-zero linear maps preserving $\L$ that are not bijective.

\begin{example}
\label{singularpreserver}
Each linear map $A \mapsto AX$ where $X \in M_n(K)$ clearly preserves $\L$  and $\leq_\L$ (by the definitions of $\L$ and $\leq_\L$); taking $X$ to be singular yields a map that is not bijective. Dually, the map $A \mapsto XA$ preserves $\R$.
\end{example}

\begin{corollary}\label{cor_L-preserver_form}
Each map of the form
\begin{equation}
\label{L_maps_over_alg_cl_fields}
A \mapsto PAX \mbox{ where }P \in{\rm GL}_n(K), \, X \in M_n(K)
\end{equation} 
preserves $\L$ and $\leq_\mathcal{L}$.	
\end{corollary}
\begin{proof} This follows from Theorem \ref{transformations} and Example \ref{singularpreserver}.
\end{proof}  
The main aim of this subsection is to prove that if $K$ is algebraically closed, or more generally if $K$ has roots for all polynomials of degree exactly $n$, then the maps of the form in \eqref{L_maps_over_alg_cl_fields} are the only maps which preserve $\L$ and/or $\leq_\mathcal{L}$. Before proceeding to do this, we remark that if polynomials of degree $n$ do \emph{not} have roots in $K$ then there are non-bijective maps preserving $\L$ having quite another structure, as we shall see in Section~\ref{nonbijectiveexamples} below.

We need some auxiliary lemmas to prove the main result of this subsection. 

\begin{lemma} \label{n_matrices_smaller_rank}
Let $n,k \geq 2$ and $K$ be a field in which every polynomial of degree exactly $n$ has a root. Let $B_1,\dots, B_n \in M_n(K)$ be matrices of rank $k$ lying in the same $\L$-class. Then some non-trivial linear combination of the $B_i$s has rank less than $k$.
\end{lemma}

\begin{proof}
Since $B_1$ has rank $k$ we may write $B_1 = C_1 X$, where $C_1 \in M_{n\times k}(K)$, $X \in M_{k\times n}(K)$ are matrices of rank $k$. Since the matrices $B_i$ ($i=2,\ldots,n$) are $\L$-related to $B_1$, it follows that we have $B_i = C_i X$ for all $i=1,\ldots,n$, where $C_i \in M_{n\times k}(K)$ and $\rk B_i=\rk C_i=k$. It therefore suffices to find $\lambda_1, \dots, \lambda_n \in K$ such that $\rk(\sum_{i=1}^n\lambda_i C_i) < k$.

Let $u_j^i$ be the $j$th column of $C_i$, ($i=1,\ldots,n$, $j=1,\ldots, k$). If $u_j^1, u_j^2, \dots, u_j^n$ are linearly dependent for some $j$, then we have $\lambda_1 u_j^1+ \lambda_2 u_j^2 + \cdots + \lambda_n u_j^n = 0$ for some $\lambda_i$ not all zero, in which case it is clear that $\rk(\sum_{i=1}^n\lambda_i C_i) < k$.

Suppose then that for each $j=1, \ldots, k$ the vectors $u_j^1, u_j^2, \dots, u_j^n$ are linearly independent. For $j=1,2$, let $A_j \in M_n(K)$ be the invertible matrix with $i$th column $u_j^i$. Consider the polynomial $f(x) = \det(A_1 + x A_2) \in K[x]$. The coefficient of $x^n$ and constant term of $f$ are equal to $\det(A_2)$ and $\det(A_1)$, respectively. Since $A_1$ and $A_2$ are invertible, this means $f$ has degree exactly $n$ and a non-zero constant term, and so by the assumption on $K$  there is a non-zero root, say $\lambda \in K\setminus \{0\}$ with $f(\lambda)=0$. This shows that the columns of $A_1 + \lambda A_2$ are linearly dependent. Thus there exist $\lambda_1, \ldots, \lambda_n \in K$ not all zero and $\lambda \neq 0$ such that
$$\lambda_1 (u_1^1+\lambda u_2^1) + \cdots + \lambda_n (u_1^n+\lambda u_2^n) = 0.$$
But this means
$\sum_{i=1}^n \lambda_i u_1^i = - \lambda \sum_{i=1}^n \lambda_i u_2^i$, in other words,  the first column of $\sum_{i=1}^n\lambda_i C_i$ is equal to a multiple of the second column. Thus  $\rk(\sum_{i=1}^n\lambda_i C_i) < k$, as desired.
\end{proof}

\begin{example}
The conclusion of Lemma \ref{n_matrices_smaller_rank} need not hold if we replace $B_1,\ldots, B_n$ by a strictly smaller collection of $\L$-equivalent matrices in $M_n(K)$, as we shall now show.
Let $u, v \in K^n$ be two linearly independent vectors. For $i = 1, 2, \dots, n - 1$, let $B_i$ denote the $n \times n$ matrix (of rank $2$) with row $i$ equal to $u$, row $i+1$ equal to $v$, and all remaining rows $0$. It is easy to see that each non-trivial linear combination of $B_1, \dots, B_{n - 1}$ is also of rank $2$.
\end{example}

\begin{corollary}
\label{rank1orless}
Let $n \geq 1$, let $K$ be a field in which every polynomial of degree exactly $n$ has a root, and let $T\colon M_n(K) \rightarrow M_n(K)$ be a linear map preserving $\L$. If $A$ is a matrix of rank $1$, then $T(A)$ has rank at most $1$.
\end{corollary}

\begin{proof}
Let $u \in K^n$ be a fixed non-zero vector, and for $i=1,\ldots, n$ let $B_i^u$ denote the matrix whose $i$th row is $u$ with all remaining rows equal to $0$. Since $T$ preserves $\L$, the matrices $T(B_i^u), i = 1,\ldots, n,$ have the same row space, and hence  in particular the same rank $k$. If $k \geq 2$ then $n \geq 2$ and  Lemma \ref{n_matrices_smaller_rank} applies to give $\rk(\sum_{i=1}^n(\lambda_i T(B_i^u))) < k$ for some $\lambda_1, \dots, \lambda_n \in K$ not all zero. However, since the matrix $\sum_{i=1}^n\lambda_i B_i^u$ is $\L$-equivalent to each $B_i^u$, we obtain a contradiction. This shows that the rank of each $T(B_i^u)$ is at most $1$. Since every rank $1$ matrix is $\L$-equivalent to a matrix of the form $B_1^u$ for some non-zero  vector $u$, the result follows.
\end{proof}

The previous corollary plays a key role in proving that (in particular) for algebraically closed fields $K$, the linear maps preserving $\L$ have the form prescribed in \eqref{L_maps_over_alg_cl_fields}. To describe the other key ingredients we require another definition. For each subspace $V$ of $K^n$ let
$$\rho(V) :=\{A \in M_n(K): {\rm Row}_K(A) \subseteq V\}.$$

The proof of the following proposition is omitted since it can be directly verified from the definitions.
\begin{proposition} \label{rho_properties}
	Let $V$ be a subspace of $K^n$, and let $v_1, \dots, v_k$ be a basis for $V$ (hence $\dim V = k$). Then
\begin{itemize}
	\item[(i)]$\rho(V) = \{ \sum_{j=1}^k  x_j^T v_j: x_1, \ldots, x_k \in K^n\}$;
	\item[(ii)] $\rho(V)$ is a subspace of $M_n(K)$;
	\item[(iii)] $\rho(V)$ has  dimension $nk$ with the basis
	$$\{ e_i ^Tv_j: 1 \leq i \leq n, 1 \leq j \leq k\},$$
	where $e_1, \ldots,  e_n$ denotes the standard basis of $K^n$.
\end{itemize}
\end{proposition}

The  next statement follows immediately from the above.
\begin{proposition} \label{prop:UV}
Let $U,V$ be subspaces of $K^n$ such that $K^n=U\oplus V$. Then $M_n(K)=\rho(U)\oplus \rho(V).$
\end{proposition}

\begin{proof}
	We show first that $\rho(U) \cap \rho(V) = \{0_{n\times n}\}$. Indeed, assume the contrary, that is, that $A \in \rho(U) \cap \rho(V)$ for some $A \in M_n(K)$. Then ${\rm Row}_K(A) \subseteq U$ and ${\rm Row}_K(A) \subseteq V$. Hence $A = 0_{n \times n}$, since $U$ and $V$ share no vectors in common apart from the zero vector.
		Furthermore, by Proposition \ref{rho_properties}(iii), $\dim \rho(U) + \dim \rho(V) = n(\dim U + \dim V) = n^2$. Thus $M_n(K)=\rho(U)\oplus~\rho(V).$
\end{proof}

We shall use the following two technical lemmas, the first of which holds without any assumptions on~$K$. Let $T: M_n(K) \rightarrow M_n(K)$ be a linear map. We shall need to consider the set
$$V_T:=\{v \in K^n: v \mbox{ is a row of some } A\in M_n(K) \mbox{ with } T(A)=0_{n \times n} \}.$$ 
\begin{lemma}
\label{lem:rho}
Let $K$ be a field and $T: M_n(K) \rightarrow M_n(K)$ be a linear map preserving $\L$. Then
\begin{itemize}
\item[(i)] the set $V_T$ is a subspace of $K^n$;
\item[(ii)] $\rho(V_T)={\rm ker}(T)$;
\item[(iii)] if $U$ is a subspace of $K^n$ with $K^n=V_T \oplus U$ then
$$M_n(K) = {\rm ker}(T) \oplus \rho(U),$$
and in particular $T$ is injective on $\rho(U)$.
\end{itemize}
\end{lemma}

\begin{proof}
(i) Let $v,w \in V_T$. Thus $v$ is a row of $A$ and $w$ is a row of $B$ where $T(A)=T(B)= 0_{n \times n}$. For $\lambda \in K$ it is clear that $\lambda v$ is a row of $\lambda A$ and $T(\lambda A)=0_{n \times n}$, showing that $\lambda v \in V_T$. Since $T$ preserves $\L$ we may assume (by permuting rows and applying Proposition~\ref{Okn}(i)) that $v$ is the first row of $A$ and $w$ is the first row of $B$. Now $T(A+B) = 0_{n \times n}$ and $A+B$ has first row $v+w$, so $v+w \in V_T$. Thus, $V_T$ is a subspace of~$K^n$.

(ii) 
It is clear from the definitions and (i) that ${\rm ker}(T) \subseteq \rho(V_T)$. We prove the other inclusion. Let $v_1, \ldots , v_k$ be a basis for $V_T$. By Proposition \ref{rho_properties}(iii), $\{e_i^Tv_j: 1 \leq  i \leq n, 1 \leq j \leq k\}$ is a basis for $\rho(V_T)$. Hence it suffices to show that each matrix $e_i^Tv_j \in {\rm ker}(T)$. Furthermore,  for each fixed $j$, it suffices to show that $e_i^Tv_j \in {\rm ker}(T)$ for some $i$, since $T$ preserves $\L$ and all the matrices $\{e_i^Tv_j : i = 1, \dots, n\}$ are $\L$-equivalent.

For $n = 1$ it is clear, so let $n \geq 2$. Since $v_j \in V_T$ there exists a matrix $A$ with some row equal to $v_j$, say row $p$, and $T(A)=0_{n \times n} $. The matrix $A'$ obtained from $A$ by adding $v_j$ to row $i$ for some $i \neq p$ then has the same row space as $A$ that is, $A' \L A$. Since $T$ preserves $\L$ this gives $T(A')=T(A)=0_{n \times n}$.  By the linearity of $T$ we see that the matrix $e_i^T v_j = A'-A$ satisfies $T(e_i^Tv_j)=T(A'-A)=T(A') - T(A)=0_{n \times n}$.

Part (iii) is a direct consequence of (ii) and Proposition~\ref{prop:UV}.
\end{proof}

\begin{remark}
An interesting algebraic consequence of Lemma~\ref{lem:rho} is that the kernel of any linear $\L$-preserver on $M_n(K)$ is a principal left ideal of $M_n(K)$ (regarded either  as a non-commutative ring, or as a multiplicative semigroup). Indeed, by the lemma
the kernel has the form $\rho(V_T)$ where $V_T$ is a subspace of $K^n$ and hence has dimension at most $n$. If we choose a matrix $X$ whose rows form a spanning set for $V_T$ then it is easy to verify that $\rho(V_T)$ is exactly the principal left ideal $M_n(K) X$. 
\end{remark}

For each row vector $u \in K^n$ write $L_u$ to denote the $\L$-class of $M_n(K)$ whose elements have row space $\langle u \rangle$. Thus $L_0=\{0_{n \times n}\}$, and for each non-zero $u$, the class
$$L_u \ = \ \{x^Tu: x \in K^n, x \neq 0\} \ = \ \rho(\langle u \rangle) \setminus \lbrace 0_{n \times n} \rbrace$$
 is a set of rank $1$ matrices. Corollary \ref{rank1orless} tells us that when $K$ contains a root for every polynomial of degree $n$, a linear map preserving $\L$ must map each $L_u$ into some $L_w$. The following lemma helps  to  make precise the ways in which this can happen.

\begin{lemma}
\label{lem:LutoLw}
Let $n \geq 1$ and let $K$ be a field in which every polynomial of degree exactly $n$ has a root. Let $T: M_n(K) \rightarrow M_n(K)$ be a linear map preserving $\L$. Let $K^n = V_T \oplus U$ and $\{u_1,\ldots, u_s\}$ be a basis for $U$. Then the following properties hold.
\begin{itemize}
\item[(i)] For $v \in K^n$, $T(L_v)=\{0_{n \times n} \}$ if and only if $v \in V_T$. 
\item[(ii)] For each $i=1,\ldots, s$, there exists a non-zero vector $w_i$ such that $T(L_{u_i})=L_{w_i}$.  The vectors $w_i$ are determined up to scalar factors.
\item[(iii)] The subspace $W:=\langle w_1, \ldots, w_s\rangle$ (which does not depend upon the choice of non-zero vectors from part (ii)) has dimension $s$.
\item[(iv)] $T(\rho(U))=\rho(W)$. 
\item[(v)] For each vector $w_i$ given in part (ii), the map $\varphi_i\colon K^n \rightarrow K^n$ defined by $\varphi_i(x)=y$ whenever $T(x^Tu_i)=y ^Tw_i$, is a linear isomorphism.
\item[(vi)] For each $i=1,\ldots, s$, there exists $c_i \in K\setminus \{0\}$ such that $\varphi_1 = c_i \varphi_i$.
\end{itemize}
\end{lemma}

\begin{proof}
(i) If $T(L_v)=\{0_{n \times n}\}$, then in particular $T(A)=0_{n \times n}$ where $A$ is the matrix with all rows equal to $v$, and so $v \in V_T$.  Conversely, if $v \in V_T$ we have $L_v   \subseteq \rho(\langle v \rangle) \subseteq \rho(V_T) ={\rm ker}(T)$,  where the inclusions follow from the definition of $\rho$ and the equality is given by part (ii) of Lemma \ref{lem:rho}. Thus $T(L_v)=\{0_{n \times n}\}$.

(ii) For each $i=1, \ldots, s$ it follows from Corollary \ref{rank1orless} that $T(L_{u_i}) \subseteq L_{w_i}$ for some $w_i \in K^n$. Recalling that $L_0 = \{0_{n \times n}\}$, part (i) gives that each $w_i$ is non-zero. Thus
$$T(L_{u_i} \cup \{0_{n \times n}\}) \subseteq L_{w_i} \cup \{0_{n \times n}\} \neq \{0_{n \times n}\}$$
and, by Lemma \ref{LR-properties-rank-1},  $L_{u_i} \cup \{0_{n \times n}\}$ and $L_{w_i} \cup \{0_{n \times n}\}$ are linear subspaces of the same dimension $n$. By part (iii) of Lemma \ref{lem:rho} we have that $T$ is injective on $\rho(U)$. Since $L_{u_i} \cup \{0_{n \times n}\} \subseteq \rho(U)$ it now follows that $T(L_{u_i}) = L_{w_i}$. It is clear from the definition that $L_x=L_y$ if and only if $x$ is a non-zero scalar multiple of $y$, and so the elements $w_i$ are determined up to scalar multiplication.

(iii) Suppose that $\sum_{i=1}^s \lambda_i w_i = 0$ for some $\lambda_1, \dots, \lambda_s \in K$. For each $i$ let $B_i$ denote the matrix  in $L_{w_i}$ with all rows equal to $w_i$. By part (ii) we have $T(L_{u_i})=L_{w_i}$, and hence there exists an element of $L_{u_i}$ mapping to $B_i$ under $T$. Since each $u_i \neq 0$, we have that $L_{u_i} = \{x^Tu_i: x \in K^n, x \neq 0\}$ and hence for each $i$ there exists a non-zero vector $x_i$ such that $T(x_i^Tu_i)=B_i$. Thus 
$$T\left(\sum_{i=1}^{s} \lambda_i x_i^Tu_i\right) =\sum_{i=1}^{s} \lambda_i T(x_i^Tu_i) = \sum_{i=1}^{s} \lambda_i B_i = 0_{n \times n}.$$ 
Now, $\sum_{i=1}^{s} \lambda_i x_i^Tu_i \in \rho(U) \cap {\rm ker}(T)$ and so, by part (iii) of Lemma \ref{lem:rho},
$$\sum_{i=1}^{s} \lambda_i x_i^Tu_i = 0_{n \times n}.$$
Since the $u_i$ are linearly independent and the $x_i$ are non-zero, it is straightforward to verify that $\lambda_i=0$ for all $i$.

(iv) Recall that $U$ has basis $\{u_1, \ldots, u_s\}$. By Proposition \ref{rho_properties}(iii) it follows that $\rho(U)$ has basis
$$\{e_i^Tu_j: 1 \leq i \leq n , 1 \leq j \leq s\}$$ and for all $j = 1, 2, \dots, s$, $L_{u_j} \cup \{0_{n \times n}\} = \rho(\langle u_j \rangle)$ has basis
$\{e_i^Tu_j: 1 \leq i \leq n\}$, where $e_1, \ldots,  e_n$ denotes the standard basis of $K^n$. Thus
$$T(\rho(U)) = \langle T(e_i^Tu_j): 1 \leq i \leq n , 1 \leq j \leq s\rangle  = \langle  L_{w_1} \cup \cdots \cup L_{w_s} \rangle \subseteq \rho(W).$$
By Lemma~\ref{lem:rho}(iii) $T$ is injective on $\rho(U)$, and so the result will follow by observing that  $\rho(U)$ and $\rho(W)$ have the same dimension. By assumption $U$ has dimension $s$, and by part (iii) $W$ also has dimension $s$. It now follows from Proposition \ref{rho_properties}(iii) that  $\rho(U)$ and $\rho(W)$ both have dimension $ns$, giving  $T(\rho(U))=\rho(W)$.
  
(v) Recall that for a non-zero $u \in K^n$, $L_{u} = \{x^Tu: x \in K^n, x \neq 0\}$. Thus it follows from the definition of $w_i$ (in part (ii)) that a map $\varphi_i$
 with the given property is well defined.
Let $x,v,y,z \in K^n$ with $T(x^Tu_i)=y^Tw_i$ and  $T(v^Tu_i)=z^Tw_i$. By linearity of $T$ we have $T((x+v)^Tu_i)=(y+z)^Tw_i$ and $T((\lambda x)^T u_i)=(\lambda y)^T w_i$. Thus $\varphi_i(x+v) = \varphi_i(x) + \varphi_i(v)$ and $\varphi_i(\lambda x)= \lambda \varphi_i(x)$.  By definition $\varphi_i(x)=0$ if and only if $T(x^Tu_i)=0_{n \times n}$. Since $x^Tu_i \in \rho(U)$  and $\rho(U) \cap {\rm ker}(T)=\{0_{n \times n}\}$ (by Lemma \ref{lem:rho} (iii)), we find that $\varphi_i(x)=0$ if and only if $x=0$. Thus $\varphi_i$ is a linear injection, and hence a linear isomorphism.
 
(vi) Clearly we can take $c_1=1$. Suppose then that $2 \leq i \leq s$. Let $x \in K^n \setminus \{0\}$ and consider
 $$T(x^T(u_1+u_i)) = T(x^Tu_1) + T(x^Tu_i) = \varphi_1(x)^T w_1 + \varphi_i(x)^T w_i.$$
By Corollary~\ref{rank1orless}, 
 the image of $x^T(u_1+u_i)$ under $T$ must have rank $0$ or $1$. Since $x$ is non-zero and both $\varphi_1$ and $\varphi_i$ are linear isomorphisms, we must have that $\varphi_1(x) \neq 0 \neq \varphi_i(x)$. Since $w_1$ and $w_i$ are linearly independent, in order for $\varphi_1(x)^T w_1 + \varphi_i(x)^T w_i$ to have rank $0$ or $1$, we require that $\varphi_1(x)=c_x \varphi_i(x)$ for some $c_x \in K \setminus \{0\}$. We show that this scalar is independent of $x$.

If $y=\lambda x$ for  $\lambda \neq 0$, then
$$\varphi_1(y)=\lambda \varphi_1(x) = \lambda c_x \varphi_i(x) = c_x \varphi_i(y)$$
and so we deduce that $c_x=c_y$.

Suppose then that $x$ and $y$ are linearly independent. Since $\varphi_i$ is an isomorphism, we see that $\varphi_i(x)$ and $\varphi_i(y)$ are linearly independent too. Then
\begin{eqnarray*}
c_y\varphi_i(y) + c_x \varphi_i(x) &=& \varphi_1(y)+\varphi_1(x) \ = \ \varphi_1(y+x)\\
&=&c_{y+x} \varphi_i(y+x) \ = \ c_{y+x} \varphi_i(y) + c_{y+x} \varphi_i(x),
\end{eqnarray*}
 giving $c_y=c_x=c_{y+x}$. Thus for all $x,y \in K^n \setminus \{0\}$ we found that $c_x=c_y \neq 0$. In other words, $\varphi_1 = c \varphi_i$ for some $c \in K\setminus \{0\}$.
\end{proof}

We are now in position to prove the main result of this section.

\begin{theorem}\label{thm:Lsing}
Let $n \geq 1$ and let $K$ be a field in which every polynomial of degree $n$ has a root. The linear maps preserving the $\L$ relation on $M_n(K)$ are precisely those of the form $A \mapsto PAX$, where $P \in {\rm GL}_n(K)$ and $X \in M_n(K)$.
\end{theorem}

\begin{proof} We have already observed in Corollary \ref{cor_L-preserver_form} that each such map preserves $\L$. We show that if $T\colon M_n(K) \rightarrow M_n(K)$ is a linear map preserving $\L$, then there exist $P \in {\rm GL}_n(K)$ and $X \in M_n(K)$ such that $T(A)=PAX$ for all $A \in M_n(K)$. Let $V_T \subseteq K^n$ be the subspace defined  before Lemma \ref{lem:rho}, and let $U$ be a subspace satisfying $K^n=V_T \oplus U$. Let $\{v_1,\ldots , v_k\}$ be a basis of $V_T$ and $\{u_1, \ldots, u_s\}$ (where $s=n-k$) be a basis of $U$. By Lemma \ref{lem:LutoLw}(ii) there exist non-zero vectors $w_1,\ldots, w_s \in K^n$ with $T(L_{u_i})=L_{w_i}$. For $i=1, \ldots, s$, let $\varphi_i$ be the linear isomorphisms described in Lemma \ref{lem:LutoLw}(v). By Lemma \ref{lem:LutoLw}(vi) there exist non-zero scalars $c_i$ such that $\varphi_1  =  c_i \varphi_i$. Since $\{u_1, \ldots, u_s, v_1, \ldots, v_k\}$ is a basis for $K^n$, we can define a linear map $\chi: K^n \rightarrow K^n$ by  $\chi(u_i)=c_i^{-1} w_i$ for $i=1,\ldots, s$ and $\chi(v_i)= 0$ for $i=1, \ldots , k$. 

Since $\varphi_1: K^n \rightarrow K^n$  is a linear isomorphism there exists a unique $P\in {\rm GL}_n(K)$ with $\varphi_1(y)=yP^T$ for all $y \in K^n$. Likewise, since $\chi: K^n \rightarrow K^n$ is a linear map there exists $X \in M_n(K)$ such that $\chi(y)=yX$ for all $y \in K^n$.

We claim that $T(A)=PAX$ for all $A \in M_n(K)$. Since both sides of the equation are linear in $A$, it will suffice to show that $T(A) = PAX$ for all $A$ in some
basis for $M_n(K)$. Since $\{u_1, \ldots, u_s, v_1, \ldots, v_k\}$ is a basis for $K^n$, we have a basis $\{ e_j^T u_i \mid 1 \leq j \leq n, 1 \leq i \leq s \} \cup \{ e_j^T v_i \mid 1 \leq j \leq n, 1 \leq i \leq k \}$ for $M_n(K)$.
Now for $1 \leq j \leq n$ and $1 \leq i \leq k$ we have  
$$P(e_j^T v_i)X = P e_j^T (v_i X) = P e_j^T \chi(v_i) =  P e_j^T 0 = 0_{n \times n} = T(e_j^T v_i),$$
since $e_j^T v_i \in \rho(V_T) = {\rm ker}(T)$ by Lemma \ref{lem:rho}(ii).
Furthermore, for $1 \leq j \leq n$ and $1 \leq i \leq s$ we have 
\begin{align*}
P(e_j^T u_i)X = (P e_j^T) (u_i X) = \varphi_1(e_j)^T \chi(u_i) = \ &c_i^{-1} \varphi_1(e_j)^T w_i \\
&= \varphi_i(e_j)^T w_i = T(e_j^T u_i),
\end{align*}
using Lemma~\ref{lem:LutoLw} parts (v) and (vi).
\end{proof}

By a left-right dual arguments (replacing row spaces with with column spaces throughout the above) one obtains:
\begin{theorem} \label{thm:Rsing}
Let $n \geq 1$ and let $K$ be a field in which every polynomial of degree $n$ has a root. The linear maps preserving the $\R$ relation on $M_n(K)$ are precisely those of the form $A \mapsto XAP$, where $P \in {\rm GL}_n(K)$ and $X \in M_n(K)$.
\end{theorem}

\begin{remark}
For $K=\mathbb{R}$  every polynomial of odd degree has a root in $\mathbb{R}$. Thus in odd dimensions Theorems \ref{thm:Lsing} (respectively, Theorem \ref{thm:Rsing}) completely describes the linear preservers of $\L$ (respectively, $\R$) on $M_n(\mathbb{R})$. We shall see
below (Example~\ref{petrovic}) that when $n$ is even there are $\L$-preservers on $M_n(\mathbb{R})$ of other kinds.
\end{remark}

\subsection{Some examples}\label{nonbijectiveexamples} Next we provide some examples of non-bijective linear  $\L$-preservers that do not fit the conditions of the above theorems. We shall use the following lemma to construct these examples, and also later to give a general description of $\H$-preservers (Theorem~\ref{hinvert}).

\begin{lemma} \label{ex_mapping_for_nonsingular_matrices}
Let $n \geq 2$ and matrices $C_1, C_2, \dots, C_n \in M_n(K)$ be a basis for a subspace of $M_n(K)$ whose non-zero elements lie in single $\L$-class $\L_X$. Let $T\colon M_n(K)\rightarrow M_n(K)$ be the linear map defined by $T(A) = \sum_{i = 1}^n a_{i1}C_i$ for all $A = (a_{ij}) \in M_n(K)$. Then:
\begin{itemize}
\item[(i)] $T$ preserves $\L$;
\item[(ii)] if $\rk X >1$, then $T$  is not of the form given in (\ref{L_maps_over_alg_cl_fields});
\item[(iii)] if $\rk X =n$, then $T$ also preserves $\H$.
\end{itemize}
\end{lemma}

\begin{proof}
(i) To show that $T$ preserves the $\mathcal{L}$ relation, consider matrices $A = (a_{ij}), B = (b_{ij}) \in M_n(K)$ such that  $A \L B$. Then the conditions $a_{11} = a_{21} = \ldots = a_{n1} = 0$ and $b_{11} = b_{21} = \ldots = b_{n1} = 0$ either both hold or both do not hold. If they both hold, then $T(A) = T(B) = 0_{n \times n}$. Otherwise, $T(A)$ and $T(B)$ are both contained in $\L_X$. In both cases we have $T(A) \mathcal{L} \,T(B)$. Hence, $T$ preserves $\L$.

(ii) Let $\rk X >1$ and suppose for a contradiction that there exist $P \in GL_n(K)$  and $X \in M_n(K)$ such that for all $A \in M_n(K)$, $T(A) = PAX$. Then there exists $Y \in M_n(K)$ such that $YX$ has rank $1$ (for example, if row $i$ of $X$ is non-zero, then the matrix with first row equal to the standard basis element $e_i$ and all other rows equal to zero has this property). But now $T(P^{-1}Y)$ has rank $1$, contradicting the fact that the image of $T$ lies in $\L_X \cup \{0\}$. 

(iii) If $X$ has rank $n$, then  $\L_X = \H_X = GL_n(K)$, and the argument given in part (i) demonstrates that $T$ preserves $\H$.
\end{proof}

Of course, there is an obvious dual statement which yields a construction of maps preserving $\R$ (and where appropriate also $\H$). The following is a very concrete example of the construction in Lemma~\ref{ex_mapping_for_nonsingular_matrices}.

\begin{example}\label{lpreserveexample}
Let $T\colon M_2(\mathbb{R}) \rightarrow M_2(\mathbb{R})$ be the following linear map:
$$T 
\begin{pmatrix}
a & b \\
c & d
\end{pmatrix}
\ = \ 
\begin{pmatrix}
a & a - c\\
c & a + c
\end{pmatrix}
 \ = \  
a \begin{pmatrix}
1 & 1\\
0 & 1
\end{pmatrix} + 
c \begin{pmatrix}
0 & -1\\
1 & 1
\end{pmatrix}.$$

Note that $\det\!\left(
\begin{smallmatrix}
a & a - c \\
c & a + c
\end{smallmatrix}\right) = a^2 + c^2 \geq 0$ with equality if and only if  $a = c = 0$.
Hence, by Lemma~\ref{ex_mapping_for_nonsingular_matrices}, $T$ preserves $\L$ and $\H$ but is not of the form given in (\ref{L_maps_over_alg_cl_fields}).
\end{example}

The following example is based on a construction of Petrovi\'{c} \cite{petrovic}.
\begin{example}\label{petrovic}
Let $K$ be $\mathbb{R}$ (or indeed any ordered field), and let $n$ be even. Consider the matrices:
$$C_{2k-1} = E_{2k-1,1}+E_{2k,2} \textrm{  and  } C_{2k}=E_{2k,1}-E_{2k-1,2}$$
for $k=1, \ldots, \frac{n}{2}$ where $E_{i,j}$ denotes the element of $M_n(K)$ with $1$ in position $(i,j)$ and zeros elsewhere.

It can be directly verified that these matrices are linearly independent and generate a subspace $V$ in which all non-zero matrices have rank $2$; see
 \cite[Proposition 1]{petrovic} in the case $K = \mathbb{R}$ and the proof given there goes through unchanged for any ordered field (since
ordered fields have the property that $x^2+y^2=0$ only if $x=y=0$).
Moreover, all the matrices $C_i$, and hence all matrices in $V$, have row spaces contained in the span of the row vectors $e_1$ and $e_2$.
Since the non-zero matrices in $V$ have rank $2$, their row space must be equal to this span, so they must all be $\L$-related. Thus, the
map $T: M_n(K) \rightarrow M_n(K)$ given by
	$$T(A) \ =
	\ \sum_{i = 1}^n a_{i1}C_i \ = \ 	
	\left(\begin{array}{ccccc}
	a_{1,1}& -a_{2,1} & 0&\cdots & 0\\
	a_{2,1} &a_{1,1} & 0 & \cdots & 0\\
	a_{3,1} & -a_{4,1} & 0&\cdots & 0\\
	a_{4,1} & a_{3,1} & 0&\cdots & 0\\
	\vdots &\vdots & \vdots & \vdots & \vdots\\
	a_{2m-1,1} &-a_{2m,1} & 0 & \cdots & 0\\
	a_{2m,1} &a_{2m-1,1} & 0 & \cdots & 0\\
	\end{array}\right)$$
for all $A = (a_{ij}) \in M_n(K)$,
satisfies the conditions of Lemma~\ref{ex_mapping_for_nonsingular_matrices}
and hence is an $\L$-preserver not of the form obtained in Theorem \ref{thm:Lsing}. 
\end{example}

The following class of examples was considered by Botta \cite[Theorem 3 and proof thereof, on page 48]{Botta}.

\begin{example} \label{ex_irreducible}
Suppose that there exists an irreducible polynomial $f(x) \in K[x]$ of degree $n \geq 2$. Let $C$ be any matrix in $M_n(K)$ whose minimal polynomial is $f(x)$. (For example, one could take $C$ to be the companion matrix of $f(x)$). Then $I, C, C^2, \dots, C^{n-1}$ satisfy
$$\det(\lambda_1I + \lambda_2C + \lambda_3C^2 + \dots + \lambda_nC^{n-1}) = 0 \iff \lambda_1 = \lambda_2 = \ldots = \lambda_n = 0,$$
and hence by Lemma \ref{ex_mapping_for_nonsingular_matrices} the linear transformation $T$ defined as $T(A) = \sum_{i = 1}^n a_{i1}C^{i-1}$, where $A = (a_{ij}) \in M_n(K)$, preserves $\L$ and $\H$, and $T$ is not of the form obtained in Theorem \ref{thm:Lsing}.
\end{example}

\begin{corollary}
Let $K={\mathbb Q}$ be the field of rationals. Then for each $n\ge 2$  there exist non-bijective linear $\L$ and $\H$-preservers that do not fit the conditions of  Theorem \ref{thm:Lsing}.
\end{corollary}

\section{The $\H$ relation}

In this section we consider linear maps on $M_n(K)$ preserving the $\H$-relation. A linear map on $M_n(K)$ is called an \textit{invertibility preserver}
(or by some authors, a \textit{non-singularity preserver}) if it preserves the set of invertible matrices. Our main result is that non-zero $\H$-preservers coincide exactly with invertibility preservers. When the field is algebraically closed, it is well known that the only invertibility preservers are the maps of Lemma \ref{transformations}. Over an arbitrary field, linear maps preserving invertibility have been fully described by de Seguins Pazzis \cite{dpazzis} in the following theorem:

\begin{theorem}\cite[Theorem 2]{dpazzis} \label{pazzisMain} 
Let $n \geq 2$, $K$ be any field, and $T\colon M_n(K) \rightarrow M_n(K)$ be a linear non-singularity preserver. Then:

(i) either $T$ is bijective and then there exist $P, Q \in GL_n(K)$ such that $T(A) = PAQ$ for all $A \in M_n(K)$ or $T(A) = PA^TQ$ for all $A \in M_n(K)$;

(ii) or there exist an $n$-dimensional subspace $V$ of $M_n(K)$ contained in $GL_n(K) \cup \lbrace 0_{n \times n} \rbrace$, an isomorphism $\alpha\colon K^n \to V$, and a non-zero $x \in K^n$ such that
$$
T(M) = \alpha(Mx^T) \ \ \forall M \in M_n(K) \ \ \textrm{ or } \ \ T(M) = \alpha(M^T x^T) \ \ \forall M \in M_n(K).
$$
\end{theorem}

 Our strategy is to prove directly that non-zero $\H$-preservers also preserve invertibility, and then use 
 Theorem \ref{pazzisMain} to show that every invertibility preserver also preserves $\H$. First we prove the following lemma.
 
\begin{lemma} \label{non-singular-maps}
 Let $K$ be any field and $R: M_n(K) \rightarrow M_r(K)$ be a linear map such that $R(GL_n(K)) \subseteq GL_r(K)$. Then $n \leq r$.
 \end{lemma}
 
 \begin{proof}
 Let $k$ be such that $r^2-k$ is the dimension of the image of $R$ (that is, the rank of $R$). Then, by the rank-nullity theorem, the kernel of $R$ has dimension $n^2-r^2+k$. Let $V$ be an $(r^2-r)$-dimensional subspace of $M_r(K)$ consisting entirely of singular matrices (for example, the set of matrices with first row equal to $0$). Since $\Im(R)$ has codimension $k$ in $M_r(K)$ and $V$ has dimension $r^2-r$, the intersection $\Im(R) \cap V$ has dimension at least $r^2-r-k$.

Consider the subspace $U=R^{-1}(\Im(R) \cap V)$ of $M_n(K)$. The image $R(U)$ is contained in $V$, so it consists of singular matrices. Since $R(GL_n(K)) \subseteq GL_r(K)$, the space $U$ must therefore consist of singular matrices. The greatest dimension of a subspace of singular matrices in $M_n(K)$ is $n^2-n$ (by for example \cite[Theorem 4(a)]{dpazzis}), so we must have ${\rm dim}(U) \leq n^2 - n$. On the other hand, since $U$ is defined as the preimage under $R$ of a subspace of dimension at least $r^2-r-k$, it contains the kernel of $R$ (which has dimension $n^2 - r^2+k$) and has image with dimension at least $r^2 - r -k$, so the rank-nullity theorem applied to the restriction of $R$ to $U$ gives ${\rm dim}(U) \geq (n^2-r^2+k) +(r^2-r-k) = n^2-r$. Thus we find that $n^2-n \geq {\rm dim}(U) \geq n^2-r$, giving $n \leq r$.
 \end{proof}

\begin{theorem}\label{hinvert}
\label{preservesnonsing}
Let $K$ be any field and $T: M_n(K) \rightarrow M_n(K)$ be a linear map. Then $T$ preserves $\H$ if and only if either $T=0$ or $T$ preserves invertibility.
\end{theorem}

\begin{proof}
For the direct implication, suppose for a contradiction that $T$ preserves $\H$, but is non-zero and does not preserve invertibility. Since
$GL_n(K)$ is a single $\H$-class we have $T(GL_n(K)) \subseteq \H_X$ for some matrix $X$ of rank $r$, where $0 < r < n$ because $T$ is non-zero and does not preserve
invertibility. Without loss of generality, we may assume that $X$ is equal to the partial identity $I_n(r):={\rm diag}(\underbrace{1, \dots, 1}_r, \underbrace{0, \dots, 0}_{n - r})$. Indeed, if not, then there exist $P, Q \in GL_n(K)$ such that $PXQ=I_n(r)$,
and using Lemma \ref{transformations}(i) we may replace $T$ with the map $M_n(K) \rightarrow M_n(K), A \mapsto PT(A)Q$.

It is easy to see that $\H_{I_n(r)}$ consists of the set of $n \times n$ matrices with an $r \times r$ invertible matrix in the top left corner, and zeros elsewhere. Thus $T$ determines a linear map $R: M_n(K) \rightarrow M_r(K)$ with $R(GL_n(K)) \subseteq GL_r(K)$. By Lemma \ref{non-singular-maps}, $n \leq r$, which is a contradiction.

For the converse implication, we use the results of de Seguins Pazzis \cite{dpazzis}. It is clear that the zero map preserves $\H$.
By Theorem \ref{pazzisMain}(i) the bijective invertibility preservers are the standard maps shown in Lemma \ref{transformations} to be $\H$-preservers. By Theorem \ref{pazzisMain}(ii) the non-bijective invertibility preservers have one of the following forms:

\begin{eqnarray}
T(M) &=& \alpha(Mx^T) \ \ \forall M \in M_n(K) \ \textrm{ or } \ \label{pazzisEq1}\ \\ T(M) &=& \alpha(M^T x^T) \ \ \forall M \in M_n(K), \label{pazzisEq2}
\end{eqnarray}
where $V$ is an $n$-dimensional subspace of $M_n(K)$ contained in $GL_n(K) \cup \lbrace 0_{n \times n} \rbrace$, $x$ is a non-zero element of $K^n$ and
$\alpha : K^n \to V$ is an isomorphism. If we set $C_i = \alpha(e_i)$ for each $i$, then the $C_i$s form a basis for $V$ 
(and thus satisfy the conditions in Lemma~\ref{ex_mapping_for_nonsingular_matrices} above, recalling that $GL_n(K)$ is an $\L$-class of $M_n(K)$) and equations \eqref{pazzisEq1} and \eqref{pazzisEq2} may be rewritten as:
\begin{eqnarray*}
T(M)&=& \sum_{i = 1}^n (Mx^T)_iC_i \; \; \forall M \in M_n(K) \;\;\mbox{ or}\\  T(M) &=& \sum_{i = 1}^n (M^T x^T)_iC_i \; \;\ \forall M \in M_n(K).
\end{eqnarray*}
Let $P$ be any invertible matrix such that $Pe_1^T=x^T$. In the first case, $T$ decomposes as the composition of the $\H$-preserving (by Lemma \ref{transformations}) map $M \mapsto MP$, with the $\H$-preserving (by Lemma \ref{ex_mapping_for_nonsingular_matrices}) map  $f\colon A \mapsto \sum_{i = 1}^n (A e_1^T)_iC_i = \sum_{i = 1}^n a_{i,1}C_i $, where $A=(a_{ij}) \in M_n(K)$. The second case is similar, with $T$ decomposing as the composition of the $\H$-preserver $M \mapsto M^T P$, and map $f$ from the first case. Thus in all cases $T$ preserves $\H$.
\end{proof}

A precise description of the maps which preserve invertibility, and therefore of the non-zero $\H$-preservers, is given in Theorem \ref{pazzisMain}, which is \cite[Theorem 2]{dpazzis}.

In the case of bijective maps, we once again obtain a statement which mirrors a result in the anti-negative semifield case  \cite{GutermanJohnsonKambites}, and therefore applies to semifields in general.

\begin{theorem}\label{bijectiveh}
Let $\S$ be a semifield and $T\colon M_n(\S) \rightarrow M_n(\S)$ be a bijective linear map. Then the following are equivalent:
\begin{itemize}
\item[(i)] $T$ preserves $\mathcal{H}$;
\item[(ii)] $T$ preserves $\leq_\mathcal{H}$;
\item[(iii)] there exist $P, Q \in GL_n(\S)$ such that either $T(A) = PAQ$ for all $A \in M_n(\S)$ or $T(A) = PA^TQ$ for all $A \in M_n(\S)$.
\end{itemize}
\end{theorem}

\begin{proof}
The case where $\S$ is an anti-negative semifield follows from \cite[Corollary~4.4]{GutermanJohnsonKambites} so by
Proposition~\ref{dichotomy} we may assume that $\S$ is a field $K$.  That (iii) implies (ii) is Lemma~\ref{transformations}, while that (ii) implies (i)
is immediate from the definition. Finally, if (i) holds, so $T$ preserves $\H$, then by Theorem~\ref{hinvert} $T$ preserves invertibility,
and since $T$ is bijective Theorem \ref{pazzisMain}(i) ensures that it has the form given in (iii).
\end{proof}

Over an algebraically closed field, or more generally in a field with enough roots, we obtain a very simple description of the $\H$-preservers.
\begin{corollary} \label{main_H_Th}
Let $n \in \mathbb{N}$ and let $K$ be a field in which every polynomial of degree $n$ has a root. Let $T: M_n(K) \rightarrow M_n(K)$ be a linear map. Then $T$ preserves $\H$ if and only if one of the following holds:
\begin{itemize}
\item $T=0$;
\item there exist $P, Q \in GL_n(K)$ such that $T(A)=PAQ$ for all $A \in M_n(K)$;
\item there exist $P, Q \in GL_n(K)$ such that $T(A)=PA^TQ$ for all $A \in M_n(K)$.
\end{itemize}
\end{corollary}
\begin{proof}
If $n=1$ then every linear map $T : M_1(K) \to M_1(K)$ clearly has one of the required forms, so suppose $n \geq 2$.
If $T$ preserves $\H$ then by Theorem~\ref{hinvert} either $T=0$ or $T$ preserves invertibility. In the latter case, if $T$ is
bijective then it has one of the required forms by Theorem \ref{pazzisMain}(i). If $T$ is not bijective then by Theorem \ref{pazzisMain}(ii) it has the form 
\begin{equation}\label{pazzis}
T(M) = \alpha(Mx^T) \ \ \forall M \in M_n(K)  \ \textrm{ or }  \ T(M) = \alpha(M^T x^T) \ \ \forall M \in M_n(K),
\end{equation}
where $V$ is an $n$-dimensional subspace of $M_n(K)$ contained in $GL_n(K) \cup \lbrace 0_{n \times n} \rbrace$, $x$ is a non-zero element of $K^n$ and
$\alpha : K^n \to V$ is an isomorphism. We prove that such a subspace $V$ cannot exist. Indeed, if it does, choose a basis $\lbrace B_1, \dots, B_n \rbrace$ for $V$. Then by Lemma~\ref{n_matrices_smaller_rank}
some non-trivial combination of the basis elements has rank strictly less than $n$, which contradicts either the linear independence of the
basis or the fact that $V$ is contained in $GL_n(K) \cup \lbrace 0_{n \times n} \rbrace$.

The converse follows from Lemma~\ref{transformations} and the obvious fact that the zero map preserves $\H$.
\end{proof}

As a corollary, we obtain a slight strengthening of a theorem of Botta \cite[Theorem 2]{Botta}, which to the best of our
knowledge has not appeared in the literature before. (The original form requires the slightly stronger hypothesis that the field is algebraically closed.)

\begin{corollary}
Let $n \in \mathbb{N}$ and let $K$ be a field in which every polynomial of degree $n$ has a root. If $T : M_n(K) \to M_n(K)$ is linear and 
preserves invertibility, then $T$ is bijective and also preserves singularity.
\end{corollary}
\begin{proof}
If $T$ preserves invertibility then clearly it is non-zero and by Theorem~\ref{hinvert} it preserves $\H$, so it has either the
second or the third form given in Corollary~\ref{main_H_Th}. It is easy to see that each of these forms gives a bijective map which
also preserves singularity.
\end{proof}

The situation is more complicated in the case of non-bijective $\H$-preservers over a field which is not algebraically closed. Indeed, as noted in \cite{Semrl}, the representation of complex numbers, quaternions, and octonions, respectively, by matrices over the real numbers provide examples of non-zero linear non-bijective maps on $M_2(\mathbb{R})$, $M_4(\mathbb{R})$ and $M_8(\mathbb{R})$, respectively, that preserve invertibility, and hence, the $\H$ relation. 
By applying the celebrated ``1,2,4,8 Theorem'' of Bott, Kervaire and Milnor \cite{BottMilnor,Kervaire},
de Seguins Pazzis \cite[Proposition 8]{dpazzis} showed that invertibility preservers in $M_n(\mathbb{R})$ are bijective for all $n$ except $2$, $4$ and $8$. Combining with Theorem~\ref{hinvert} we have:
\begin{corollary}\label{cor248}
If $n \notin \lbrace 2, 4, 8 \rbrace$ then every linear $\H$-preserver on $M_n(\mathbb{R})$ is either zero or bijective.
\end{corollary}

\section{The $\J$ relation}
Two elements of $M_n(K)$ are $\J$-related precisely if they have the same rank and the $\mathcal{J}$-order corresponds exactly to the natural order on ranks (see \cite[Lemma 2.1]{Okninsk}). For $r=0,\ldots, n$ we denote the set of matrices of rank $r$ by~$\J_r$.
\begin{lemma} \label{sum_of_2}
Let $K$ be a field and $n$ a natural number.
\begin{itemize}
\item[(i)] For $0 \leq r \leq k \leq n$,  every rank $r$ matrix in $M_n(K)$ can be written as a sum of two rank $k$ matrices.
\item[(ii)] For $3 \leq r \leq n$, every rank $r$ matrix in $M_n(K)$ can be written as a sum of two rank $r-1$ matrices.
\end{itemize}
\end{lemma}
\begin{proof}
A rank $r$ matrix $A$ can be written as $A=PI_n(r)Q$, where $I_n(r)={\rm diag}(\underbrace{1, \dots, 1}_r, \underbrace{0, \dots, 0}_{n - r})$ is a partial identity and $P, Q\in {\rm GL}_n(K)$.  Thus it suffices to show that each partial identity $I_n(r)$ can be written as a sum of two matrices of appropriate rank. In both parts, we consider separately the case where $K$ is the $2$-element field.

(i) Suppose first that $K$ has more than $2$ elements. Then for all $0 \leq r \leq k \leq n$ we may choose any $h \in K \setminus \{0, 1\}$ 
and write
\begin{align*}
I_n(r) = \ &{\rm diag}(\underbrace{1 - h, \dots, 1 - h}_r, \underbrace{1, \dots, 1}_{k - r}, 0, \dots, 0)\ + \\ & {\rm diag}(\underbrace{h, \dots, h}_r, \underbrace{-1, \dots, -1}_{k - r}, 0, \dots, 0).
\end{align*}
It is clear that these two diagonal matrices are of rank $k$.

If $K$ is the $2$-element field then for all $0 \leq r \leq k < n$, we have
\[
I_n(r) = \left(I_n(r) + \sum_{i = 1}^k E_{i, i + 1}\right) + \left(\sum_{i = 1}^k E_{i, i + 1}\right),
\]
where $E_{x,y}$ denotes the matrix with $1$ in the $(x,y)$ position and $0$ elsewhere.
It is clear that each of the bracketed expressions in the equation above is a matrix with exactly $k$ non-zero rows, and straightforward to verify that both matrices have rank $k$. In   the case where $k = n$, we have that
\begin{align*}
I_n(r) = &  \left(I_n(r) + I_n(1) + E_{n,1} + \sum_{i = 1}^{n-1} E_{i, i + 1}\right) + \left(I_n(1) + E_{n,1} +\sum_{i = 1}^{n-1} E_{i, i + 1}\right).
\end{align*}
Again, it is easy to see that the two matrices have $n$ non-zero rows and a straightforward calculation reveals that  both have rank $n$. This completes the proof of part (i).

(ii) If $K$ has more than $2$ elements, then for all $3 \leq r \leq n$ and choosing any $h \in K \setminus \{0, 1\}$ we have 
$$I_n(r) = {\rm diag}(\underbrace{1 - h, \dots, 1 - h}_{r-2}, 1, 0, \underbrace{0, \dots, 0}_{n-r}) + {\rm diag}(\underbrace{ h, \dots, h}_{r-2}, 0, 1, \underbrace{0, \dots, 0}_{n-r}).$$
and it is clear that these two matrices have rank $r-1$.

If $K$ is the $2$-element field then for all $3 \leq r \leq n$ we may write
$$I_n(r) =\left(\begin{array}{cc} A_1& 0\\0&B_1 \end{array}\right) + \left(\begin{array}{cc} A_2& 0\\0&B_2 \end{array}\right)$$
where
$$ 
A_1=\begin{pmatrix}
1 & 1 & 0 \\
1 & 1 & 0 \\
0 & 0 & 1
\end{pmatrix}, A_2=\begin{pmatrix}
0 & 1 & 0 \\
1 & 0 & 0 \\
0 & 0 & 0
\end{pmatrix},$$
and $B_1, B_2$ are $(n-3) \times (n-3)$ matrices of rank $r-3$ which sum to $I_{n-3}(r-3)$. Such matrices exist by part (i). Since $A_1$ and $A_2$ have rank $2$, it is easy to see that this writes $I_n(r)$ as a sum of two rank $r-1$ matrices. 
\end{proof}

Once again, the following statement applies to semifields in general, because the result we establish here for fields mirrors the existing
result for anti-negative semifields.
\begin{theorem}
\label{thm_J_bij}
Let $\S$ be a semifield and $T\colon M_n(\S) \rightarrow M_n(\S)$ be a bijective linear map. Then the following are equivalent:
\begin{itemize}
\item[(i)] $T$ preserves $\mathcal{J}$;
\item[(ii)] $T$ preserves $\leq_\mathcal{J}$;
\item[(iii)] $T$ preserves $\H$;
\item[(iv)] $T$ preserves $\leq_\H$;
\item[(v)] there exist $P, Q \in GL_n(K)$ such that either $T(A) = PAQ$ for all $A \in M_n(K)$ or $T(A) = PA^TQ$ for all $A \in M_n(K)$.
\end{itemize}

\end{theorem}

\begin{proof}
The equivalence of (iii), (iv) and (v) is Theorem~\ref{bijectiveh} above. The case where $\S$ is an anti-negative is \cite[Corollary~4.4]{GutermanJohnsonKambites}. It will thus suffice to show the
equivalence of (i), (ii) and (v) in the case that $\S = K$ is a field. That (v) implies (ii) is given by Lemma~\ref{transformations}, while
that (ii) implies (i) is immediate from the definitions. It remains only to show that (i) implies (v). We shall do this by proving that (i)
implies $T$ preserves rank $1$ matrices, at which point Theorem \ref{rank-1-preservers} gives (v).

Suppose then that $T$ preserves $\J$. Then for each $r=0, \ldots, n$ we have $T(\J_r) \subseteq \J_{\sigma(r)}$ where $\sigma$ maps $\{0, \ldots, n\}$ to itself. In fact, since $T$ is bijective, we must have that   $\sigma$ is a permutation and $T(\J_r) = \J_{\sigma(r)}$ for each $r$.  Since $T$ is linear we have $\sigma(0)=0$.  We claim that $\sigma(1)=1$. If $n=1$ this is clear, so suppose from now on that $n \geq 2$. 

Suppose for contradiction that $k=\sigma^{-1}(1)>1$. (That is, $T(\J_k) = \J_1$, where $k > 1$.) Let $Y$ be a non-zero matrix of rank strictly less than $k$. By Lemma \ref{sum_of_2}(i) we can write $Y$ as a sum of two matrices of rank $k$ and hence $T(Y)$ as a sum of two matrices of rank $1$. It follows 
that $T(Y)$ has rank at most $2$. Since we have already observed that $T(\J_0)=\J_0$ and $T(\J_k)=\J_1$, we must therefore have that $T(Y)$ has rank $2$. If $k>2$ this immediately gives a contradiction, since multiple $\J$-classes (namely, $\J_1, \ldots, \J_{k-1}$) will be mapped to $\J_2$. Thus we must have $k=2$, or in other words $T(\J_2)=\J_1$, and by the argument just given, $T(\J_1)=\J_2$.
 
For $n \geq 3$, Lemma \ref{sum_of_2}(ii) allows us to express a matrix of rank $3$ as a sum of two matrices of rank $2$. Since
$T(\J_2) = \J_1$ this means that the image of any rank $3$ matrix can be expressed as the sum of two rank $1$ matrices, and hence has rank at most $2$. Thus, $T(\J_3) \subseteq \J_0 \cup \J_1 \cup J_2$. But this gives a contradiction since we have already seen that $T$ permutes $\J$-classes and $T(\J_0)=\J_0$, $T(\J_2)=\J_1$ and $T(\J_1)=\J_2$ and so we cannot map $\J_3$ to any $\J$-class of lower rank.

For $n=2$, we use two more arguments to obtain a contradiction, the first of which applies only in the case where the field contains at least five elements. Let $\Omega$ be a set of $2 \times 2$ matrices, and consider the following question: do there exist two matrices $A$ and $B$ (not necessarily in $\Omega$) such that the set of scalars $\lbrace \lambda \mid A + \lambda B \in \Omega \rbrace$ has cardinality in $\lbrace 1, 2 \rbrace$? Clearly this is a purely linear property of $\Omega$ as a subset of $M_2(K)$, so the answer to this question must be the same for $\Omega = \J_1$ and $\Omega = \J_2 = T(\J_1)$ since $T$ is a linear isomorphism.

For $\Omega = \J_1$ the answer to this question is positive. For example, taking $A = \left(\begin{array}{cc} 1 & 0 \\ 0 & 1 \end{array}\right)$
and
$B = \left(\begin{array}{cc} 0 & 1 \\ 1 & 0 \end{array}\right)$ we have that $A+\lambda B \in \J_1$ if and only if $\lambda = 1$ or $\lambda = -1$ so the given set has cardinality $1$ (if the field has characteristic $2$) or $2$ (otherwise). 

On the other hand, if the field has at least $5$ elements then for $\Omega = \J_2 = GL_2(K)$ the answer is negative. Indeed, the non-singularity condition
$\det(A+\lambda B) \neq 0$ is quadratic in $\lambda$, so either the determinant is identically zero (in which case the cardinality of the given set is $0$) or else it has at most two roots (in which case at least $3 = 5-2$ values of $\lambda$ do not satisfy it, so the cardinality of the given set is at least $3$).

Thus, we cannot have $T(\J_1) = \J_2$ where the field has $5$ or more elements. There remain only the three finite fields with strictly fewer than $5$ elements and in each case the fact that $T(\J_1) \neq \J_2$ can be seen by simple counting arguments. Over the $2$-element or $4$-element fields the number of non-zero $2 \times 2$ matrices is odd, so $\J_1$ and $\J_2$ (whose union gives the set of all such non-zero matrices) cannot have the same cardinality. Over the $3$-element field it is well known and easy to calculate that $|\J_1| = 32$ and $|\J_2| = 48$. Thus  in each of these remaining small cases there cannot be any bijective map  taking $\J_1$ to $\J_2$, even without the linearity condition.
\end{proof}

\begin{theorem}
\label{thm_J_nonbij}
Let $K$ be any field. Then a linear map $T\colon M_n(K) \to M_n(K)$ preserving $\J$ is either the zero map or a bijection.
\end{theorem}

\begin{proof}
If $T(\J_1) = \J_0$, then $T$ is the zero map, since each matrix is a sum of rank one matrices. Suppose then that there exists $A \in \J_1$ with $T(A) \neq 0_{n \times n}$. By part (i) of Lemma \ref{sum_of_2}, for each $1 \leq r \leq n$ there exist matrices $B_r$ and $C_r$ of rank $r$ such that $A=B_r+C_r$. Applying $T$ we find that
$$0_{n \times n} \ \neq \ T(A) \ = \ T(B_r)+T(C_r)$$
and hence for each $1 \leq r \leq n$ we see that $T(\J_r) \neq \J_0$. Since $T$ preserves $\J$ we deduce that the kernel of
$T$ is trivial, so by the rank-nullity theorem $T$ is bijective.  
\end{proof}

\begin{corollary}
Let $K$ be any field. Then a linear map $T\colon M_n(K) \to M_n(K)$ preserving $\leq_\J$ is either the zero map or a bijection.
\end{corollary}

\end{document}